\documentclass[11pt]{amsart}
\usepackage{amsmath,amsfonts,amssymb}

\usepackage{latexsym}
\usepackage{eufrak}
\usepackage{graphicx}
\usepackage{epsf}
\usepackage{psfrag}
\usepackage{tikz}

\theoremstyle{plain}
\newtheorem{lemma}{Lemma}[section]
\newtheorem{theorem}[lemma]{Theorem}

\newtheorem{proposition}[lemma]{Proposition}

\theoremstyle{definition}
\newtheorem{remark}[lemma]{Remark}
\newtheorem{definition}[lemma]{Definition}

\numberwithin{equation}{section}

\begin{document}

\title[Directed polymers in a random environment with a defect line]{Directed polymers in a random environment with a defect line}

\author{Kenneth S. Alexander and G\"{o}khan Y{\i}ld{\i}r{\i}m}
\address{Department of Mathematics KAP 108 \\
University of Southern California\\
Los Angeles, CA  90089-2532 USA}
\email{alexandr@usc.edu}
\email{gyildiri@usc.edu}
\thanks{This research was supported by NSF grant DMS-0804934.}


\maketitle

\begin{abstract} 
We study the depinning transition of the $1+1$ dimensional directed polymer in a random environment with a defect line. The random environment consists of i.i.d. potential values assigned to each site of $\mathbb{Z}^2$; sites on the positive axis have the potential enhanced by a deterministic value $u$. We show that for small inverse temperature $\beta$ the quenched and annealed free energies differ significantly at most in a small neighborhood (of size of order $\beta$) of the annealed critical point $u_c^a=0$.  For the case $u=0$, we show that the difference between quenched and annealed free energies is of order $\beta^4$ as $\beta\to 0$, assuming only finiteness of exponential moments of the potential values, improving existing results which required stronger assumptions.

\end{abstract}

\section{Introduction.}
\subsection{Physical Motivation}

The directed polymer in a random environment (DPRE) models a one-dimensional object interacting with disorder. The $1+1$ dimensional version of the model first appeared in the physics literature in \cite{HH} as a model for the interface in two-dimensional Ising models with random exchange interaction. Since then it has been used in models of various growth phenomena: formation of magnetic domains in spin-glasses \cite{HH}, vortex lines in superconductors \cite{Nel}, turbulence in viscous incompressible fluids (Burger turbulence)\cite{Bur}, roughness of crack interfaces \cite{HHR}, and the KPZ equation \cite{KPZ}.

A related problem is the competition between extended and point defects as reflected in pinning phenomena, arising for example in the context of high-temperature superconductors \cite{BSL, CMWT}. On a lattice this can be described by a random potential, typically i.i.d.~at each lattice site, representing the point defects, with an additional fixed potential $u$ added for sites along some line, representing the extended defect.  The polymer must choose between roughly following the extended defect, or finding the best path(s) through the point defects.  As $u$ is decreased, one expects a depinning transition at some critical $u_c$ where the polymer ceases to follow the extended defect.  

In the (nonrigorous) physics literature, there have been disagreeing predictions as to whether $u_c=0$. Kardar \cite{K2} examined this problem numerically and found that $u_c>0$ for the $1+1$ dimensional DPRE with defect line. On the other hand, Tang and Lyuksyutov in \cite{TL} argued that the same model satisfies $u_c=0$, and claimed that $u_c>0$ only above $1+1$ dimensions. Their conclusion was supported by Balents and Kardar \cite{BK}, numerically and via a functional renormalization group analysis, and later by Hwa and Natterman \cite{HN} in another renormalization group analysis.  It is hoped that a mathematically rigorous analysis can eventually resolve the question.  

\subsection{Mathematical Formulation of the Problem}

The DPRE in $1+d$ dimensions is formulated as follows.  Let $P_\nu$ be the distribution of the symmetric simple random walk (SSRW) $S=\{S_j, j\geq 0\}$ on $\mathbb{Z}^d$ with initial distribution $\nu$, and let $E_\nu$ be the corresponding expectation. We write $P_x,E_x$ when $\nu=\delta_x$, and $P,E$ for $P_0,E_0$.   The polymer configuration is represented by the path $\{(j, S_j)\}_{j=1}^{n}$ in $\mathbb{N} \times \mathbb{Z}^d$.  The random environment, or bulk disorder, is given by mean zero, variance one i.i.d.~random variables $V=\{v(i,x) : i \geq 1, x \in \mathbb{Z}^d\}$ with law denoted $Q$ satisfying 
\begin{equation}\label{expm}
\Lambda(\beta)=\log E^{Q}[e^{\beta v(i,x)}]< \infty \quad \text{for all}  \quad \beta \in \mathbb{R}.
\end{equation}
The Hamiltonian for paths $s$ is
\[
  H_N(s)=\sum_{j=1}^{N}v(j, s_j), 
\]
and the \textit{quenched polymer measure} $\mu_N^{\beta, q}$ is defined in the usual Boltzmann-Gibbs way:
\begin{equation} \label {eq: BG}
\frac{d\mu_N^{\beta, q}}{dP}(s)=\frac{1}{Z_N^{\beta, q }} e^{\beta H_N(s)},
\end{equation}
where $\beta>0$ is the inverse temperature and $Z_N^{\beta, q}=E_0\left[e^{\beta H_N(S)}\right]$
is the \textit{quenched partition function.}

The first rigorous mathematical work on directed polymers in $1+d$ dimensions was done by Imbrie and Spencer \cite{IS}, proving that in dimension $d\geq 3$ with Bernoulli disorder and small enough $\beta,$ the end point of the polymer scales as $n^{1/2},$ i.e. the polymer is diffusive. Bolthausen \cite{Bolt} considered the nonnegative martingale $W_n^{\beta, q}=Z_n^{\beta, q}/E^Q[Z_n^{\beta, q}]$ and observed that the almost sure limit $W_{\infty}=\lim_{n\to \infty}W_n^{\beta, q}$ is subject to a dichotomy:
there are only two possibilities for the positivity of the limit, $Q(W_{\infty}>0)=1$ (known as \emph{weak disorder}) or $Q(W_{\infty}=0)=1$ (known as \emph{strong disorder}),
because the event $\{W_{\infty}=0\}$ is a tail event. Bolthausen also improved the result of Imbrie and Spencer to a central limit theorem for the end point of the walk, which means that in $d\geq 3$ entropy dominates at high enough temperature, in that the polymer behaves almost as if the disorder were absent. 
Comets and Yoshida \cite{CSY1, CY}, showed that there exists a critical value $\beta_{c}=\beta_c(d, v)\in [0,\infty]$ with $\beta_c =0, $ for $d=1,2$ and  $0<\beta_c \leq \infty$ for $ d \geq 3$, such that 
$Q(W_{\infty}>0)=1$  if $ \beta \in \{0\}\cup(0,\beta_c)$, and $Q(W_{\infty}=0)=1$  if $ \beta >\beta_c.$ In particular, for the $1+1$ dimensional case we consider here, disorder is always strong.  See \cite{CSY2} for a survey.

There has been substantial investigation of pinning models in which disorder is present only in the defect line $\{0\} \times \mathbb{N}$; see (\cite{G1, G2} and \cite{T}) for surveys. In such models (which we call \emph{pinning models with defect-line potential}), the energy gains from pinning compete only with the entropy loss inherent in the class of pinned paths.  Here, by contrast, we enhance the potential in the DPRE by a fixed amount $u$ at each site of the defect line, so that energy gains from the enhancement for pinned paths also compete with the possibility of better energy gains from the potential $v(i,x)$ along depinned paths compared to pinned ones.  
Specifically, we define the Hamiltonian and the \emph{quenched polymer measure} by
\begin{eqnarray}
\label{hamiltonian}
 H_N^{u}(s)&=& \sum_{j=1}^{N}(v(j, s_j)+u1_{s_j=0}) = H_N(s)+uL_N(s),
\end{eqnarray}
\begin{equation}
\frac{d\mu_N^{\beta, u, q}}{dP}(s)=\frac{1}{Z_N^{\beta, u, q}}e^{\beta H_N^{u}(s)},
\end{equation}
where 
\[
  L_N(s) = \sum_{j=1}^{N} 1_{s_j=0}, \quad Z_N^{\beta, u, q}=E_0\left[e^{\beta H_N^{u}(S)}\right]
\]
are the \emph{local time} and the \textit{quenched partition function}, respectively. Here $P$ is the distribution of the SSRW with $S_0=0$.

In general for 
a partition function $Z$, the restriction to a set $\Omega$ of SSRW paths will be denoted $Z(\Omega)$; we add a subscript $\nu$ when the SSRW has initial distribution $\nu$, and include $V$ as an argument of $Z$ when we wish to emphasize the dependence on the disorder configuration $V$.  Thus for example,
\[
  Z^{\beta, u, q}_{N,\nu}(\Omega,V):=
    E_\nu\left( e^{\beta H_N^u(S) } 1_{\Omega}(S) \right).
\]
When $\nu=\delta_x$ we write $x$ in place of $\nu$.

Our results concern only $d=1$ so we restrict to that case henceforth.
Our first result is on the existence of the quenched free energy of the model:
\begin{theorem}
\label{freeE}
For every $\beta>0$ and $u\in \mathbb{R},$
\begin{align}
f^q(\beta, u)=\lim_{N\to \infty}\frac{1}{N}\log Z_N^{\beta, u, q}=\lim_{N\to \infty}\frac{1}{N}E^Q[\log Z_N^{\beta, u, q}]
\end{align}
exists $Q$-a.s. and in $Q$-mean.
\end{theorem}

The \emph{annealed polymer measure} $\mu_N^{\beta u}$ is obtained by taking the expected value over the disorder of the quenched Boltzmann-Gibbs weight, yielding 
\begin{equation} \label{annmeas}
  \frac{d\mu_N^{\beta u}}{dP}(s)=\frac{1}{Z_N^{\beta, u}}e^{\beta uL_N(s)+\Lambda(\beta)N},
\end{equation}
where
\[
  Z_N^\gamma = E_0(e^{\gamma L_N(S)}),\quad Z_N^{\beta, u} = Z_N^{\beta u} e^{\Lambda(\beta)N} = E_0( e^{\beta uL_N(S)+\Lambda(\beta)N} )
\]
is the \emph{annealed partition function}.  Note that $\mu_N^{\beta u}$ depends only on the product $\beta u$.  Letting
\begin{eqnarray*}
  \mathrm{F}(\gamma)=\lim_{N\to \infty}\frac{1}{N} \log Z_N^\gamma,
\end{eqnarray*}
the \emph{annealed free energy} is 
\[
  f^a(\beta,u) = \lim_{N\to \infty}\frac{1}{N}\log Z_N^{\beta, u} = \mathrm{F}(\beta u) +\Lambda(\beta).
\]
Here $\mathrm{F}(\cdot)$ is the free energy of the pinning model with homogeneous defect-line potential, that is, with disorder $v \equiv 0$.

The \emph{quenched} and \emph{annealed critical points} are
\[
  u_c^q(\beta) = \inf\{u: f^q(\beta,u) > f^q(\beta,0)\}, \quad u_c^a(\beta) = \inf\{u: f^a(\beta,u)>f^a(\beta,0)\}.
\]
Note that the last inequality is equivalent to $\mathrm{F}(\beta u)>0$, so $\beta u_c^a(\beta)$ does not depend on $\beta$.  In fact,
it is standard (see \cite{G1}) that in the present situation $u_c^a(\beta)=0$ for all $\beta$ because the random walk on $\mathbb{Z}$ with distribution $P$ is recurrent.  When $u>u_c^q(\beta)$ the quenched polymer is said to be \emph{pinned}.  Note also that $f^q(\beta,u)\leq f^a(\beta,u)$ by Jensen's inequality.  

 As mentioned above, physicists have differed on the question of whether $u_c^q(\beta)=0$, for $d=1$.  One approach which at least provides a bound for $u_c^q(\beta)$ is to find a value $\Delta_0(\beta)$ such that for $u>\Delta_0(\beta)$, the quenched and annealed free energies are approximately the same; in particular this means the quenched free energy is strictly greater than $\Lambda(\beta)$ and thus also strictly greater than $f^q(\beta,0)$, meaning that $u>u_c^q(\beta)$.  We thereby obtain that $u_c^q(\beta) \leq \Delta_0(\beta)$.  This is the approach taken in \cite{A} for the pinning model with defect-line potential; in the case where the underlying process is 1-dimensional SSRW one has $\Delta_0(\beta) $ of order at most $e^{-K/\beta^2}$ for some constant $K$, for small $\beta$.  Here our main result has a similar form, but with bound $\Delta_0(\beta)$ of order $\beta$.  This larger size of $\Delta_0(\beta)$ is rooted in the larger overlap present in the DPRE---overlap is counted throughout the bulk of $\mathbb{Z}^2$, as opposed to just on the axis.  (Here by overlap we mean intersections between two independent copies of the path---see \eqref{overlap}.) We do not know whether $\Delta_0(\beta)$ of order $\beta$ is optimal; the physicists' predictions in (\cite{BK},\cite{HN},\cite{TL}) point toward $u_c(\beta)=0$.  Analogs of $u_c(\beta)=0$ were in fact proved for the randomized polynuclear growth model \cite{BSV} (see also \cite{BSSV}) and recently also for the longest increasing subsequence problem and last passage percolation \cite{BSS}.  At any rate, the theorem says in effect that the disorder alters the free energy significantly at most for $u$ in a neighborhood of size $O(\beta)$ of the annealed critical point $u_c^a(\beta)=0$.
 
 We can now state our main results.

\begin{theorem}
\label{MainTh} Consider the $1+1$ dimensional DPRE with defect line, with Hamiltonian as in (\ref{hamiltonian}). Suppose that the disorder variables
$V=\{v(i,x) : i \geq 1, x \in \mathbb{Z}^d\}$ are i.i.d.~mean zero variance one random variables which satisfy the condition \eqref{expm}. 

Then given $0<\epsilon<1,$ there exists a $K=K(\epsilon)$ as follows.
Provided that $\beta$ and $\beta u$ are sufficiently small and $u\geq K\beta,$ we have
\begin{equation} \label{samefree}
  \Lambda(\beta)+\mathrm{F}(\beta u)\geq f^q(\beta,u)\geq \Lambda(\beta) + (1-\epsilon)\mathrm{F}(\beta u).
\end{equation}
Further, for small $\beta$,
\begin{equation} \label{critbound}
  0\leq u^q_c(\beta)\leq K(\epsilon) \beta.
\end{equation}
\end{theorem}

With minor modifications, the proof of Theorem \ref{MainTh} also proves the following.

\begin{theorem} \label{beta4}
Under the hypotheses of Theorem \ref{MainTh}, there exist constants $C_1,C_2$ such that for sufficiently small $\beta$,
\begin{equation} \label{Obeta4}
  C_1\beta^4 \leq f^a(\beta,0) - f^q(\beta,0) \leq C_2\beta^4.
\end{equation}
\end{theorem}

Lacoin \cite{L} proved \eqref{Obeta4} in the case of Gaussian disorder, and proved a similar statement with an upper bound of $C_2\beta^4(1+(\log \beta)^2)$ for the general disorder we consider here.  Watbled \cite{W} extended \eqref{Obeta4} to infinitely divisible disorder.

The full strength of assumption \eqref{expm} is used only to establish the existence of the free energy for all $\beta>0$, in Theorem \ref{freeE}.  For Theorems \ref{MainTh} and \ref{beta4} we need only that $\Lambda(\beta)<\infty$ for small $\beta$.  

 In the following sections, the $K_i's$ are universal constants, except where they depend on a parameter, which is shown in parentheses. 
 
\section{Proof of Theorem~\ref{freeE}:Existence of the Free Energy}
In the case $u=0$, the existence of the quenched free energy is a consequence of the concentration of $\log Z_N^{\beta,0,q}$ around its mean, together with superadditivity of $E^Q(\log Z_N^{\beta,0,q})$ in $N$, which yields a limit for $N^{-1}E^Q(\log Z_N^{\beta,0,q})$; see \cite{CaHu1}, \cite{CSY1}.  For $u\neq 0$, though, the superadditivity fails because $E^Q(\log Z_N^{\beta,u,q})$ is inhomogeneous, in the sense that if we start paths at some $(j,x)$ instead of $(0,0)$, the distribution depends on $x$.  Let us write $Z_N(x)$ or $Z_N(x,V)$ for $Z_{N, x}^{\beta, u, q}$, and $Z_N$ for $Z_N^{\beta,u,q}$ (suppressing the $\beta,u,q$ for notational convenience), and define
\begin{align*}
&Z_N(x,y)=Z_N(x,y;V):=E_x\left[e^{\sum_{j=1}^{N}\beta (v(j, S_j)+u1_{S_j=0})} 1_{S_N=y}\right],
\end{align*}
where $P_x$ is the SSRW measure when $S_0=x$.  As we will see below, for general $u$ one can easily obtain superadditivity of $E^Q(\log Z_N(0,0))$, and the proof of concentration of $\log Z_N$ around its mean requires little change; the main task is to bound the difference between $E^Q(\log Z_N)$ and $\frac{1}{2}E^Q(\log Z_{2N}(0,0))$.

\subsection{The Constrained Model}
In the \emph{constrained model} (quenched or annealed), we restrict to paths ending at $s_N=0$, so the quenched partition function is $Z_N(0,0)$.

Due to the periodicity of SSRW, we assume that $N, M$ are even integers for this section.  Let $\theta_{n,y}$ be the space-time shift operator on the environment $V$: 
\[
  (\theta_{n,y}v)(k,x)=v(k+n,x+y).
  \]
From the Markov property of SSRW, we have
\begin{align}\label{subadd}
Z_{N+M}(0,0;V)\geq Z_{N}(0,x;V)Z_{M}(x,0;\theta_{N,0}V) \quad \text{for all } N, M, x.
\end{align}
For $N=M$, after taking logs and expectations this yields 
\begin{equation} \label{endat0}
  E^{Q}[\log Z_{N}(0,x;V)] \leq \frac{1}{2} E^{Q}[\log Z_{2N}(0,0;V)] \quad \text{for all } N,x.
\end{equation}
Similarly we obtain
\begin{align}\label{subaddit}
E^{Q}[\log Z_{N+M}(0,0;V)]&\geq E^{Q}[\log Z_{N}(0, 0;V)] + E^{Q}[\log Z_{M}(0, 0;V)].
\end{align}
This superadditivity establishes the existence of the limit
\[
  \lim_{N\to \infty}\frac{1}{N}E^{Q}[\log Z_{N}(0, 0;V)]=\sup_{N\geq 1}\frac{1}{N}E^{Q}[\log Z_{N}(0, 0;V)].
\]
It follows from (\ref{subadd}), with $x=0$, and the subadditive ergodic theorem \cite{King} that the constrained free energy exists and $Q$-a.s. constant:
\[
  f^{q, c}(\beta, u)=\lim_{N\to \infty}\frac{1}{N}\log Z_{N}(0, 0;V)=\lim_{N\to \infty}\frac{1}{N}E^{Q}[\log Z_{N}(0, 0;V) ].
  \] 
The non-randomness (a.s.)~of $f^{q, c}(\beta, u)$ is called the \textit{self-averaging property} of the quenched free energy.

\subsection{The Unconstrained Model}
Since $Z_N(0,0) \leq Z_N$, if we show 
\begin{equation}\label{con-uncon}
  E^{Q}[\log Z_{N}] \leq \frac{1}{2} E^{Q}[\log Z_{2N}(0,0)] + o(N),
\end{equation}
it follows that
\begin{equation}\label{samelimit}
  \lim_{N\to \infty}\frac{1}{N}E^{Q}[\log Z_{N}(0, 0) ] = \lim_{N\to \infty}\frac{1}{N}E^{Q}[\log Z_N ].
  \end{equation}
  
Inside the proof of (\cite{CSY1}, Proposition 2.5), the following is established for the case $u=0$:  the deviation from the mean can be expressed as a sum of martingale differences,
\[
  \log Z_N - E^Q(\log Z_N) = \sum_{j=1}^N W_{N,j},
  \]
satisfying 
\[
  E^Q\left( e^{|W_{N,j}|} \right) \leq K_0(\beta) < \infty \quad \text{for all } N,j.
\]
This proof extends to $Z_N(0,x)$ simply by restricting to paths ending at $x$, and it extends to general $u$ by adding $\beta uL_N$ to the exponent in the definition of $\hat{e}_{N,j}$ in the proof in \cite{CSY1}.
Then by (\cite{LV} Theorem 3.6), there exists $K_1(\beta,p)$ such that for all $t>0$ and all $N,x$, 
\begin{eqnarray}\label{Cmeasure}
Q\Big(\Big|\frac{1}{N}\log Z_N(0,x) -\frac{1}{N}E^{Q}[\log Z_N(0,x) ]\Big|\geq t\Big)\leq \frac{K_1(\beta,p)}{t^pN^{p/2}}.
\end{eqnarray}
We can now establish \eqref{con-uncon}.
Let $\Lambda_N=\{(i,x): 1\leq i\leq N, |x|\leq i, x-i$ even$\}$. Then using \eqref{endat0}
\begin{align} \label{EQlogbound}
E^{Q}[\log Z_{N}]
  &= E^{Q}\left[\log \left(\sum_{x: (N,x) \in \Lambda_N}e^{\log Z_{N}(0,x)}\right)\right]\notag\\
&=E^{Q}\left[\log\left( \sum_{x: (N,x) \in \Lambda_N} e^{  
  E^{Q}[\log Z_{N}(0, x)]} e^{(\log Z_{N}(0,x)-E^{Q}[\log Z_{N}(0,x)])}\right)\right]\notag\\
&\leq \frac{1}{2} E^{Q}[\log Z_{2N}(0,0)]
  +E^{Q}\left[\log\left( \sum_{x: (N,x) \in \Lambda_N}
  e^{(\log Z_{N}(0,x)-E^{Q}[\log Z_{N}(0,x)])}\right)\right]\notag\\
&\leq  \frac{1}{2} E^{Q}[\log Z_{2N}(0,0)] +  E^{Q}\Big[\log \left( (2N+1)\max_{x: (N,x) \in \Lambda_N}  
 e^{(\log Z_{N}(0,x)-E^{Q}[\log Z_{N}(0,x)])} \right)\Big]\notag\\
 & \leq  \frac{1}{2} E^{Q}[\log Z_{2N}(0,0)]+\log(2N+1)\notag\\
 &\qquad +E^{Q}\Big[\max_{x: (N,x) \in \Lambda_N} (\log Z_{N}(0,x)-E^{Q}[\log Z_{N}(0,x)])\Big]\notag\\
 & \leq  \frac{1}{2} E^{Q}[\log Z_{2N}(0,0)]+\log(2N+1)\notag\\
 &\qquad +\int_0^{\infty}Q\Big(\max_{x: (N,x) \in \Lambda_N} |\log Z_{N}(0,x)-E^{Q}[\log Z_{N}(0,x)]|\geq s \Big)\ ds.
\end{align}

For $q_N>0$ we can bound the last integral using \eqref{Cmeasure} with $p=3$:
\begin{align*}
&\int_0^{\infty}Q\Big(\max_{x: (N,x) \in \Lambda_N}  
  |\log Z_{N}(0,x)-E^{Q}[\log Z_{N}(0,x)]|\geq s\Big)\ ds\\
& \leq q_N+(2N+1)\int_{q_N}^{\infty}\max_{x: (N,x) \in \Lambda_N}Q\Big(  |
   \log Z_{N}(0,x)-E^{Q}[\log Z_{N}(0,x)]|\geq s\Big)\ ds\\
&\leq q_N+(2N+1)N\int_{q_N/N}^{\infty}\max_{x: (N,x) \in \Lambda_N}Q\Big(  |
  \log Z_{N}(0,x)-E^{Q}[\log Z_{N}(0,x)]|\geq Nt\Big)\ dt\\
& \leq q_N + 3K_1(\beta,3) N^{1/2}\int_{q_N/N}^{\infty} t^{-3}\ dt\\
&\leq q_N+ \frac{3}{2}K_1(\beta,3)N^{5/2}q_N^{-2}.
\end{align*}
Choosing $q_N=N^{5/6}$ we see that the integral on the right side of \eqref{EQlogbound} is $O(N^{5/6})$, and hence \eqref{con-uncon} holds.  Therefore so does \eqref{samelimit}.  

The Borel-Cantelli lemma, and (\ref{Cmeasure}) with $p>2$, then establish
the equality of the free energies in the original and constrained models.  

\section{Proof of Theorem~\ref{MainTh}}
\subsection{Proof Outline}
We take a block length $N$ which is a multiple (of order $\epsilon^{-2}$) of the annealed correlation length, so that the associated finite-volume annealed free energy is large.  We use the second moment method to show that on scale $N$, the quenched partition function is with high probability within a constant of the annealed one; here the condition $u\geq K(\epsilon)\beta$ allows necessary control of the overlap.  This remains true if we restrict the partition functions to a set $\Omega_N$ of paths which stay inside an $N\times 4\sqrt{N}$ box centered on the axis, and end within $\sqrt{N}/4$ of the axis.  Having paths end close to the axis facilitates concatenating a large number $L$ of the boxes together to make a length-$LN$ corridor in such a way that the corresponding partition function is approximately the product of the $L$ single-box partition functions.

Certain boxes in this corridor, though, may have very small values for the associated quenched partition function, making this product of single-box partition functions unacceptably small relative to the annealed one.  This requires re-routing the corridor through off-axis boxes in places, to avoid ``bad'' on-axis boxes; bad off-axis boxes must also be avoided in this process.  The result is a dependent percolation problem on coarse-grained scale; one needs an infinite directed path of ``good'' boxes, with most of these boxes being on-axis, where the extra potential $u$ is relevant.  We use results of \cite{GHDD}, \cite{GH} and \cite{LSS} to establish the existence of such a path.  The restriction of the quenched partition function to length-$LN$ paths following the corresponding (non-coarse-grained) corridor then provides a lower bound for the full quenched partition function at length $LN$, and taking a limit as $L\to\infty$ yields the desired result.

\subsection{Further Preliminaries}
Recall that $\mathrm{F}(\gamma)$ denotes the free energy of the homogeneous (or annealed) model with defect-line potential. As observed in (\cite{A}, equation (2.7)), $\gamma+\log E_0(e^{\gamma L_n})$ is subadditive in $n$ for all $\gamma \geq 0.$  It follows that
\begin{equation} \label{FLB1}
  E_0(e^{\gamma L_N}) \geq e^{-\gamma}e^{N \mathrm{F}(\gamma)} \quad \text{for all } N \geq 1.
  \end{equation}
  
In what follows, in service of clean notation, we omit (but implicitly assume) integer part notation for large quantities which in fact must be integers, such as $M$ in the next lemma.  

The following is essentially the same as (\cite{A}, equation (2.22)).

\begin{lemma} 
\label{UPB}
There exist $K_2,K_3>0$ such that 
\[
  \forall j\geq 1, \gamma>0,\quad  E_0(e^{\gamma L_{jM}})\leq K_2 j e^{K_3j},
  \] 
where $M=1/\mathrm{F}(\gamma )$ is the correlation length.
 \end{lemma}  
 
 For the proof of the following see \cite{G1} or \cite{G2}.
 
\begin{proposition}
\label{Fprops}
The free energy ${\rm F}(\gamma)$ has the following properties:
\begin{itemize}
\item[a)] $\mathrm{F}(\gamma)$ is $0$ on $(-\infty, 0]$ and strictly increasing and positive on $(0, \infty).$
\item[b)] for some $K_4>0$, $ \mathrm{F}(\gamma)\sim K_4\gamma^2,$ as $\gamma \to 0^+$.
\end{itemize}
\end{proposition}

For any $x\in \mathbb{Z},$ $\gamma \geq 0,$ conditioning on the hitting time of 0 yields
\begin{equation} \label{RWmarkov}
E_xe^{\gamma L_N}\leq E_0e^{\gamma(L_N+1)}.
\end{equation}
For $k>1$, conditioning on $S_{(k-1)N}$, applying \eqref{RWmarkov} and iterating we obtain
\begin{equation} \label{LL}
  E_xe^{\gamma L_{kN}}\leq  \left( E_0e^{\gamma(L_N+1)}\right)^{k}.
\end{equation}

The following is a straightforward consequence of Donsker's invariance principle.  

\begin{lemma} 
\label{donsker} 
For one dimensional SSRW, we have
\begin{align*}
&A^{\rm{forward}} := \liminf_{N\to \infty}\inf_{|x|\leq \frac{\sqrt{N}}{4}}P_x\Big(\max_{1\leq i 
\leq N}|S_i|\leq 2\sqrt{N}, |S_N|\leq \frac{\sqrt{N}}{4}\Big)>0,\\
&A^{\rm{up}} := \liminf_{N\to \infty}\inf_{|x|\leq \frac{\sqrt{N}}{4}}P_x\Big(\max_{1\leq i 
\leq N}|S_i|\leq 2\sqrt{N}, |S_N-\sqrt{N}|\leq \frac{\sqrt{N}}{4}\Big)>0,\\
&A^{\rm{down}} := \liminf_{N\to \infty}\inf_{|x|\leq \frac{\sqrt{N}}{4}}P_x\Big(\max_{1\leq i 
\leq N}|S_i|\leq 2\sqrt{N}, |S_N+\sqrt{N}|\leq \frac{\sqrt{N}}{4}\Big)>0.
\end{align*}
\end{lemma}

The proof of the following is due to S.R.S. Varadhan \cite{Var}. %[personal comm.]

\begin{lemma} 
\label{hayati}
There exists a constant $0<\epsilon_0<1,$ such that for $\gamma>0$, for all sufficiently large $N$ and $|x|\leq \frac{\sqrt{N}}{4},$
\[E_x\Big(e^{\gamma L_N}1_{\Omega_N}\Big)\geq \epsilon_0 E_x\Big(e^{\gamma L_N}\Big),\]
where 
\[
  \Omega_N=\{s: \max_{1\leq i \leq N}|s_i|\leq 2\sqrt{N}, |s_N|\leq \frac{\sqrt{N}}{4}\}.
\]
\end{lemma}

\begin{proof} 
We define a polymer measure on the space of SSRW paths:
\[\mu_{N,x}^\gamma (A):=\frac{E_x[e^{\gamma L_N}1_A]}{E_x[e^{\gamma L_N}]}.\]
Let $W(n,x)=E_x[e^{\gamma L_n}]$. 

Under $\mu_{N,x}^\gamma (\cdot)$ we have a non-stationary Markov process with transition probabilities from $z$ to $y=z\pm 1$ at time $k<N$ given by 
\begin{eqnarray} \label{transit}
\pi(z,y,k,N,\gamma )&= &\frac{E_x[e^{\gamma L_N}1_{S_k=z}1_{S_{k+1}=y}]} {E_x[e^{\gamma L_N}1_{S_k=z}]}\notag\\
&=&\frac{E_x[e^{\gamma L_k}1_{S_k=z}]E_z[e^{\gamma L_{N-k}}1_{S_1=y}]} {E_x[e^{\gamma L_k}1_{S_k=z}]E_z[e^{\gamma L_{N-k}}]}\notag\\
&=&\frac{1}{2}\frac{e^{\gamma \delta_0(y)} E_y[e^{\gamma L_{N-k-1}}]} 
  {E_z[e^{\gamma L_{N-k}}]}\notag\\
&=&\frac{e^{\gamma \delta_0(y)}}{2}\frac{W(N-k-1,y)}{W(N-k,z)}.
\end{eqnarray}
For all $z$,
\begin{equation} \label{decomp}
  W(N-k,z)=\frac{1}{2}e^{\gamma \delta_0(z+1)}W(N-k-1,z+1)+\frac{1}{2}e^{\gamma \delta_0(z-1)}W(N-k-1,z-1)
\end{equation}
while for $z\geq 1$ we have monotonicity in $z$:
\[W(N-k-1,z+1)\leq W(N-k-1,z)\ \text{and} \ W(N-k-1,1)\leq e^\gamma  W(N-k-1,0),\]
which follows from the fact that the hitting time of 0 is stochastically smaller when starting from a lower height $z \geq 0$.
Similarly for $z\leq -1,$
\[W(N-k-1,z)\leq W(N-k-1,z+1)\ \text{and} \ W(N-k-1,-1)\leq e^\gamma  W(N-k-1,0).\]
Therefore
for $z\geq 1$, the second term on the right in \eqref{decomp} is the larger one, and by \eqref{transit} we thus have
\[\pi(z,z-1,k,N,\gamma )\geq \frac{1}{2},\] 
while for $z\leq -1,$ similarly,
\[\pi(z,z+1,k,N,\gamma )\geq \frac{1}{2}.\]
Hence, the $\mu_{N,x}^\gamma $ chain can be coupled to the $P_x$ chain (i.e.~SSRW) in a such a way that the $\mu_{N,x}^\gamma $ chain is always smaller or equal in magnitude.
Therefore
\[\mu_{N,x}^\gamma (\Omega_N)\geq P_x(\Omega_N),\]
and the result then follows from Lemma~\ref{donsker}.
\end{proof}

Let \[\tau_x=\inf\{n\geq 1:S_n=x\}.\]

\begin{lemma}
\label{kusto}
Let $0<\epsilon<1$ be given. Then, for sufficiently large $N$ and $|x|\leq \frac{\sqrt{N}}{4}$, for all $\gamma>0$,
\[ E_x\Big(e^{\gamma L_N}\Big)\geq\frac{1}{2}\mathbf{P}(\xi\geq  \frac{1}{4\sqrt{\epsilon}})e^{(1-\epsilon)N\mathrm{F}(\gamma)},\]
where $\xi$ denotes a standard normal random variable.
\end{lemma}

\begin{proof}
For a given $0<\epsilon<1,$ there exists an $N_0=N_0(\epsilon)$ such that for all $N\geq N_0$ and for $0<x\leq \frac{\sqrt{N}}{4}$,
\begin{eqnarray} \label{reflection}
P_x(\tau_0\leq \epsilon N)&=&P_0(\tau_x\leq \epsilon N)\notag\\
&\geq& P_0(S_{ \epsilon N}\geq \frac{\sqrt{N}}{4})\notag\\
&\geq& \frac{1}{2}\mathbf{P}(\xi\geq  \frac{1}{4\sqrt{\epsilon}}).
\end{eqnarray}
The right side of \eqref{reflection} is also a lower bound for the left side for $-\frac{\sqrt{N}}{4}\leq x< 0$ by symmetry, and for $x=0$ after increasing $N_0$ if necessary.
Therefore, for sufficiently large $N$ and $|x|\leq \frac{\sqrt{N}}{4},$ using \eqref{FLB1} and \eqref{reflection},
\begin{eqnarray*}
 E_x\Big(e^{\gamma L_N}\Big)
&\geq& \sum_{k=x}^{\epsilon N}e^{\gamma} E_0\Big(e^{\gamma L_{N-k}}\Big)P_x(\tau_0=k)\\
&\geq& \sum_{k=x}^{\epsilon N}e^{(1-\epsilon)N\mathrm{F}(\gamma)}P_x(\tau_0 =k)\\   
&=& e^{(1-\epsilon)N\mathrm{F}(\gamma)}P_x(\tau_0 \leq \epsilon N)\\
&\geq&\frac{1}{2}\mathbf{P}(\xi\geq  \frac{1}{4\sqrt{\epsilon}})e^{(1-\epsilon)N\mathrm{F}(\gamma)}.
\end{eqnarray*}
\end{proof}

For SSRW paths $s^1,s^2,$ define the overlap
\begin{equation} \label{overlap}
B_N(s^1, s^2)=\sum_{i=1}^N1_{s^1_i=s^2_i}
\end{equation}
For independent copies $S^1,S^2$ of the Markov chain $S$, $(S^1,S^2)$ is also a Markov chain, so as a special case of \eqref{RWmarkov},
\begin{equation}
E_{(x,x^{\prime})}^{\otimes 2} e^{\gamma B_N}\leq E_{(0,0)}^{\otimes 2}e^{\gamma (B_N+1)},
\end{equation}
and as a special case of \eqref{LL},
for  $k\geq 1, \gamma \geq 0,$ and $x, x^{\prime} \in \mathbb{Z}$, we have 
\begin{equation} \label{2case}
E_{(x,x^{\prime})}^{\otimes 2} e^{\gamma B_{kN}}\leq\left( E_{(0,0)}^{\otimes 2}e^{\gamma (B_N+1)}\right)^k.
\end{equation}

We need information about the excursion length distribution of $(p,q)$-walks. First, a definition:

\begin{definition}
A $(p,q)$-walk is a random walk in which the steps $X_i$ have distribution $\mathbf{P}(X_1=b)=\mathbf{P}(X_1=-b)=p/2 \in (0, 1/2)$ and $\mathbf{P}(X_1=0)=q>0$, where $p+q=1$ and $b$ is a positive integer.
\end{definition}

Let $\bar{S}_N=S_N^1-S^2_N,$ where $S_N^1,S_N^2$ are independent SSRWs. Then $(\bar{S}_N)_{N\geq 1}$ is a $(1/2,1/2)$-walk with $b=2$, and $B_N(S^1,S^2)=L_N(\bar{S}).$

For the proof of the following, see \cite{Flr} and \cite{G1}.
\begin{proposition}
\label{pqwalk} For any $(p,q)$-walk, $p\in(0,1),$ we have
\[\mathbf{P}(\tau_0=n)\sim\sqrt{\frac{p}{2\pi}}n^{-3/2}\ \text{as} \ n\to \infty.\]
For $(1,0)$-walk, 
\[\mathbf{P}(\tau_0=2n)\sim\sqrt{\frac{1}{4\pi}}n^{-3/2}\ \text{as} \ n\to \infty.\]
\end{proposition}
Let us define 
\begin{equation}
\Phi(\beta)=\Lambda(2\beta)-2\Lambda(\beta)
\end{equation}
where $\Lambda(\beta)=\log E^{Q}[e^{\beta v(i,x)}].$

The next result is similar to (\cite{A} equation (2.40)), but specialized to the present situation.

\begin{proposition}
\label{CRL} 
Let $0<a<1$ be given. Then there exists a constant $K_5=K_5(a)>0$ such that for sufficiently small $\beta$ and $R \leq K_5\beta^{-4}$ we have  
\begin{equation} \label{abound}
  E_{(0,0)}^{\otimes 2}\left( e^{2\Phi(\beta)(B_R(S^1, S^2)+1)}-1\right)\leq a.
  \end{equation}
\end{proposition}

\begin{proof}
Let $E_i$ denote the length of the $i^{th}$ excursion of $\bar{S}=S^1-S^2$ from $0$ (that is, the time from the $(i-1)$st to the $i$th visit to 0.) Then
\begin{eqnarray*} 
P(B_R+1>k)&\leq& P(\max_{1\leq i \leq k}E_i\leq R)
= (1-P(E_1>R))^k \ \ \ \ \text{for all} \  k\geq 1.
\end{eqnarray*}
By Proposition~\ref{pqwalk}, $P(E_1>R)\sim (\pi R)^{-1/2}\ \text{as} \ R\to \infty$, so for sufficiently large $R$,
\begin{eqnarray}
P(B_R+1>k) \leq \left(1-\frac{1}{\sqrt{2\pi R}} \right)^k \ \ \ \ \text{for all} \  k\geq 1.
\end{eqnarray}
Therefore $B_R+1$ is stochastically dominated by a geometric random variable with parameter 
\begin{equation} \label{pMlower}
  p_R=(2\pi R)^{-1/2} \geq \frac{\beta^2}{\sqrt{2\pi K_5}}
\end{equation}
Therefore for $R$ large and $\beta$ small,
\begin{eqnarray} \label{GE}
E_{(0,0)}^{\otimes 2}\left( e^{2\Phi(\beta) (B_R(S^1, S^2)+1)}-1\right)\leq \frac{p_Re^{2\Phi(\beta) }}{1-(1-p_R)e^{2\Phi(\beta) }}-1,
\end{eqnarray} 
provided that  
\begin{equation}
\label{MGC}
p_R>1-e^{-2\Phi(\beta)}.
\end{equation}
To bound \eqref{GE} by the given $a$, we need
\begin{equation} \label{pMlower2}
p_R\geq \frac{a+1}{a}(1-e^{-2\Phi(\beta)}).
\end{equation}
Since $\Lambda(\beta)\sim\beta^2/2$, and hence $\Phi(\beta)\sim \beta^2$, as $\beta \to 0$, if $K_5(a)$ is taken sufficiently small, then \eqref{MGC} and \eqref{pMlower2} follow from \eqref{pMlower}.  This proves \eqref{abound} for $R \leq K_5\beta^{-4}$ with $R$ large. Since the left side of \eqref{abound} is monotone in $R$, $R \leq K_5\beta^{-4}$ alone is sufficient.

%Since $\Lambda(\beta)\sim\beta^2/2$, and hence $\Phi(\beta)\sim \beta^2$, as $\beta \to 0$, by \eqref{pMlower} it suffices for \eqref{MGC} and \eqref{pMlower2} that $u \geq K_5(a)\beta$.
\end{proof}

\subsection{The Coarse Grained Lattice $\mathbb{L}_{CG}$}
\label{coarsegraining}

In this section, we introduce a coarse grained lattice
\[\mathbb{L}_{CG}:=\{(I, J)\in \mathbb{Z}^2: I\geq 0, 0\leq J \leq I\}.\]
Note this is really a ``half lattice'' since we only consider $J\geq 0$.

Recall that the annealed correlation length is $M=1/\mathrm{F}(\beta u)$.
Let $N=k_0M$, with $k_0$ to be specified.  For notational convenience we assume that $N$ and $\sqrt{N}$ are integers. 
We use capital letters $(I, J)$ for a site in the coarse grained lattice which corresponds to the vertical window
\[R(I,J):=\{(k,l)\in \mathbb{Z}^2: k=IN, \ (J-\frac{1}{4})\sqrt{N}\leq l \leq  (J+\frac{1}{4})\sqrt{N}\}\] in the original lattice $\mathbb{Z}^2.$ 

The \textit{box} starting from the window $R(I,J)$ is the following region in $\mathbb{Z}^2:$ 
\[B(I,J):=[IN, (I+1)N]\times[(J-2)\sqrt{N}, (J+2)\sqrt{N} ].\] 

We say that there is a \textit{link} between sites $(I,J)$ and $(I+1,L)$ if 
$|L-J|\leq 1$. The link is \emph{down, forward} or \emph{up} according as $L=J-1,J$ or $J+1$.
A \textit{path} $\Gamma=\Gamma_{(I,J) \to (K,L)}$ from site $(I, J)$ to site $(K,L)$ in $\mathbb{L}_{CG}$ is a sequence of sites $(I,J)=(I_0,J_0), (I_1,J_1), \cdots, (I_N,J_N)=(K,L)$ such that there is a link between $(I_i,J_i)$ and $(I_{i+1},J_{i+1})$ for all $i<N$.
$\Gamma(I_i)$ will denote the second coordinate $J_i$ of the unique site $(I_i,J_i)$ in the path $\Gamma.$
We will use the alternate notation $\Gamma_{(I,J)}$ for $\Gamma_{(0,0) \to (I,J)}.$ 
Given paths $\Gamma^1,\Gamma^2$ from some $(I, J)$ to $(K,L),$ we say that $\Gamma^1$ is \emph{closer to the $x$-axis than} $\Gamma^2$ if 
\[\Gamma^1(I_i)\leq \Gamma^2(I_i)\ \text{for each}\ I\leq I_i\leq K.\]
Suppose each site $(I,J) \in \mathbb{L}_{CG}$ is designated as \textit{open} or \textit{closed}. 
We then say a path $\Gamma_{(I,J) \to (K,L)}$ is 
\begin{itemize}
\item[(i)] \textit{open} if its all sites are open;
\item[(ii)] \textit{maximal} if it has the maximum number of open sites among all paths from site $(I, J)$ to site $(K,L)$;
\item[(iii)] \textit{optimal} if it is the maximal path which is closest to the $x$-axis. 
\end{itemize}
$\Gamma^{\infty}_{(I,J)}$ denotes a generic infinite open path from the site $(I,J)$.
There is exactly one optimal path for given sites $(I, J)$ and $(K,L)$ and we denote it by $\Gamma^{\text{opt}}_{(I,J) \to (K,L)}.$ 

When an infinite open path from a site $(I,J)$ exists, the one which is closest to the $x$-axis among all such paths is called the \textit{infinite good path from the site} $(I,J),$ and we denote it by $\Gamma^{\text{G},\infty}_{(I,J)}.$ 
$\Gamma^{\text{G},\infty}$ denotes the infinite good path from the site $(0,0)$, when it exists.
For $0\leq I\leq K,$ $\Gamma^{\text{G},\infty}_{I\to K}$ will denote the segment of the path $\Gamma^{\text{G},\infty}$ between the sites with first coordinates $I$ and $K$.
Note that if the site $(I_0,J_0)$ is on the infinite good path from $(0,0),$ then 
\begin{equation} \label{samepath}
\Gamma^{\text{opt}}_{(0,0)\to (I_0,J_0)}=\Gamma^{\text{G},\infty}_{0\to I_0}.
\end{equation}
Given a path $\Gamma=\Gamma_{(0,0)\to(I,J)} = \{(L,J_L): L\leq I\}$ in $\mathbb{L}_{CG},$ we identify a subset $\Omega^{(I,J)}$ of the SSRW paths of length $IN$ in the following way: 
\[
  \Omega^{(I,J)}:= \Omega^{(I,J)}(\Gamma) := \bigg\{ s = \{(n,s_n)\}_{n \leq IN}: 
    s_0=0,s_{LN} \in R(L,J_L)\ \forall L\leq I, s \subset \cup_{L<I} B(L,J_L) \bigg\}.
\]
When $\Gamma^{\text{G},\infty}= \{(L,J_L^G): L\geq 0\}$ exists, for $0\leq I\leq K$ we define
\[
  \Omega^{\text{G},\infty}_{I\to K} := \bigg\{ s = \{(n,s_n)\}_{IN\leq n \leq KN}: 
    s_{LN} \in R(L,J_L^G)\ \forall I\leq L\leq K, s \subset \cup_{I\leq L<K} B(L,J_L^G) \bigg\},
\]
otherwise we define $\Omega^{\text{G},\infty}_{I\to K} := \phi$.
We define quenched probability measures on the windows $R(I,J)$, using SSRW paths associated to the optimal coarse-grained path to that window,
as follows: for $I\geq 1$ and $x\in R(I,J),$ let
\begin{equation}
  \nu_{(I,J)}^q(x):=\frac{Z^{\beta, u, q}_{IN}\left(\Omega^{(I,J)}(\Gamma^{\text{opt}}_{(0,0)\to(I,J)}) \cap \{s_{IN}=x\}\right) }
    {Z^{\beta, u, q}_{IN}\left(\Omega^{(I,J)}(\Gamma^{\text{opt}}_{(0,0)\to(I,J)})\right)}, \quad x \in R(I,J),
\end{equation}
and let $\nu^q_{(0,0)}:=\delta_0$.  The measure
\[
  \tilde{\nu}_{(I,J)}^q(x) = \nu_{(I,J)}^q((IN,JN)+x), \quad x \in R(0,0),
\]
is the translate of $\nu_{(I,J)}^q$ to $R(0,0)$.

Define the following sets of SSRW paths, corresponding to up, forward and down links in a coarse-grained path:
\[\Omega_N^{\text{up}}:=\{(s_0, \cdots, s_N): |s_0|\leq \frac{\sqrt{N}}{4}, |s_N-\sqrt{N}|\leq \frac{\sqrt{N}}{4}, |s_i|\leq 2\sqrt{N}, 1\leq i\leq N\},\]
\[\Omega_N^{\text{forward}}:=\{(s_0, \cdots, s_N): |s_0|\leq \frac{\sqrt{N}}{4}, |s_N|\leq \frac{\sqrt{N}}{4}, |s_i|\leq 2\sqrt{N}, 1\leq i\leq N\},\]
and
\begin{equation*}
\Omega_N^{\text{down}}:=\{(s_0,\cdots, s_N):|s_0|\leq \frac{\sqrt{N}}{4}, |s_N+\sqrt{N}|\leq \frac{\sqrt{N}}{4}, |s_i|\leq 2 \sqrt{N}, 1\leq i\leq N\}.
\end{equation*}
Note that the up, forward and down sets of SSRW paths start at the window $R(I,J)$, stay in the box $B(I,J)$, and end at the window $R(I+1,J+l), \ l=+1,0,-1,$ respectively. 

Of particular interest are the \emph{link partition functions}
\[
  Z^{\beta, u, q}_{N, \tilde{\nu}^q_{(I,J)}}(\Omega_N^{\text{g}},\theta_{IN,JN}(V)), \quad g=\text{up, forward, down},
\]
corresponding to SSRW paths in the box $B(I,J)$ from the window $R(I,J)$ to $R(I+1,J+l)$, with $l=1,0,-1$ according to the value of $g$. When $J=0$ and $g=$ forward, we refer to the link or partition function as \emph{on-axis}, otherwise it is \emph{off-axis}.

\subsection{Open and Closed Sites in the Coarse Grained Lattice.}

Define the filtrations
\[\mathcal{F}_{I}:=\sigma(\{v(i,x): 1\leq i \leq IN, x \in \mathbb{Z}\}), \ I\geq 1,\]
and note that the measures $\nu^q_{(I,J)}$ are $\mathcal{F}_{I}$-measurable for all $J\geq 0$.
One expects on-axis link partition functions to be larger than off-axis ones in general, and we will specify
constants $U_{\text{on}} \geq U_{\text{off}}$ which will serve as lower bounds for these partition functions, satisfying
\[
  U_{\text{on}}\leq \frac{1}{2}
    E^Q\Big(Z^{\beta, u, q}_{N, \tilde{\nu}^q_{(I,0)}}\left(\Omega_N^{\text{forward}},\theta_{IN,0}(V)\right)\ \big|\ 
    \mathcal{F}_{I}\Big)\quad Q-a.s.\  \text{for each}\ I\geq 0,
  \] 
and for $I>0,J\leq I$ and g = forward, up, down, 
\[
  U_{\text{off}}\leq \frac{1}{2}E^Q\Big(Z^{\beta, u, q}_{N, \tilde{\nu}^q_{(I,J)}}\left(\Omega_N^{\text{g}},\theta_{IN,JN}(V)\right)\ \big|\ 
    \mathcal{F}_{I}\Big)\quad Q-a.s.
\]
For $I\geq 1,$ by Lemma~\ref{kusto} and~\ref{hayati}, for sufficiently small $\beta u,$ $Q$-a.s.
\begin{align*}
E^Q&\Big(Z^{\beta, u, q}_{N, \tilde{\nu}^q_{(I,0)}}\left(\Omega_N^{\text{forward}},\theta_{IN,0}(V)\right)\ \bigg|\ \mathcal{F}_{I}\Big)\\
&=\sum_{x\in R(0,0)}\tilde{\nu}^q_{(I,0)}(x)E^{Q}\Bigg(E_x\left[e^{\beta \sum_{k=1}^{N}\Big(v(IN+k, S_k)+u1_{S_k=0}\Big)}1_{\Omega_N^{\text{forward}}}\right]\Bigg)\\
&=\sum_{x\in R(0,0)}\tilde{\nu}^q_{(I,0)}(x)e^{\Lambda(\beta)N}E_x\left[e^{\sum_{k=1}^{N}\beta u1_{S_k=0}}1_{\Omega_N^{\text{forward}}}\right]\\
&\geq\sum_{x\in R(0,0)}\tilde{\nu}^q_{(I,0)}(x)e^{\Lambda(\beta)N} \frac{\epsilon_0}{2}\mathbf{P}(\xi\geq  \frac{1}{4\sqrt{\epsilon}})e^{(1-\epsilon)N\mathrm{F}(\beta u)}\\
&\geq \frac{\epsilon_0}{2}\mathbf{P}(\xi\geq  \frac{1}{4\sqrt{\epsilon}})e^{(\Lambda(\beta)+ (1-\epsilon)\mathrm{F}(\beta u))N}.\\
\end{align*}
Hence we define 
\begin{equation}
\label{onx}
  \Theta_{\rm on} := \Theta_{\rm on}(\epsilon) := \frac{\epsilon_0}{4}\mathbf{P}(\xi \geq  \frac{1}{4\sqrt{\epsilon}}), \quad 
    U_{\rm on}:=\Theta_{\rm on}e^{(\Lambda(\beta) + (1-\epsilon)\mathrm{F}(\beta u))N}.
\end{equation}
For sufficiently small $\beta u>0,$ for all $I\geq 0, J\geq 1$ and for g = forward, up, down, by Lemma \ref{donsker} we have $Q$-a.s.
\begin{align*}
E^{Q}\Big(Z^{\beta, u, q}_{N, \tilde{\nu}^q_{(I,J)}}( \Omega_N^{\text{g}},\theta_{IN,JN}(V) )|\mathcal{F}_{I}\Big)
  &\geq E^{Q}\Big(Z^{\beta, 0, q}_{N, \tilde{\nu}^q_{(I,J)}}(\Omega_N^{\text{g}},\theta_{IN,JN}(V))|\mathcal{F}_{I}\Big)\\
&\geq e^{\Lambda(\beta)N} \sum_{x\in R(0,0)}\tilde{\nu}^q_{(I,J)}(x) P_x\left(\Omega_N^{\text{g}}\right)\\
&\geq \frac{1}{2}e^{\Lambda(\beta)N}\min(A^{\rm{forward}},A^{\rm{up}},A^{\rm{down}}).
\end{align*}
Hence we define 
\begin{equation}
\label{offx}
  \Theta_{\rm off} := \Theta_{\rm off}(\epsilon) := \frac{1}{4}\min(A^{\rm{forward}},A^{\rm{up}},A^{\rm{down}},4\Theta_{\rm on}), \quad
    U_{\text{off}}:= \Theta_{\rm off}e^{\Lambda(\beta)N}.
\end{equation}
\label{rule}
We can then define open sites inductively on $I$.  The site $(0,0)$ is called \textit{open} if
\[Z^{\beta, u, q}_{N}(\Omega_N^{\text{up}})\geq U_{\text{off}}\ \text{and}\  Z^{\beta, u, q}_{N}(\Omega_N^{\text{forward}})\geq U_{\text{on}},\]
otherwise $(0,0)$ is \emph{closed}.  
Assume that all the sites $(K,L),$ for $0\leq K<I$ and $0 \leq L\leq K$ have been defined as open or closed.
Then the site $(I,0)$ is \textit{open} if 
\[Z^{\beta, u, q}_{N, \tilde{\nu}^q_{(I,0)}}(\Omega_N^{\text{up}},\theta_{IN,0}(V))\geq U_{\text{off}} \quad \text{and} \quad 
Z^{\beta, u, q}_{N, \tilde{\nu}^q_{(I,0)}}(\Omega_N^{\text{forward}},\theta_{IN,0}(V))\geq U_{\text{on}},\]
and the site $(I,J),\ 0<J\leq I$, is \textit{open} if 
\begin{equation} \label{opensite}
Z^{\beta, u, q}_{N, \tilde{\nu}^q_{(I,J)}}(\Omega_N^{\text{g}},\theta_{IN,JN}(V))\geq U_{\text{off}},  \quad \text{g = up, forward, down},
\end{equation}
otherwise $(I,J)$ is \emph{closed}.
Note the inductive definition is necessary because the previously defined open/closed values determine the optimal path from $(0,0)$ to $(I,J),$ which determines $\tilde{\nu}^q_{(I,J)}$.  Let $X_{(I,J)} = 1_{\{(I,J)\text{ is open}\} }$.
  
\subsection{Second Moment Method and Probability of an Open Site.}
\label{siteden}
 We will use the second moment method to show the probability of a closed site is small.
In general, for $Y$ a random variable with finite mean and variance, and $\theta,\epsilon\in (0,1)$, 
by Chebychev's Inequality we have
\begin{equation} \label{2ndmom}
  P((1-\theta)EY \leq Y  \leq (1+\theta)EY) \geq 1-\epsilon,
\end{equation}
provided that 
\begin{equation} \label{2ndmom2}
\frac{Var(Y)}{( EY)^2}\leq \theta^2\epsilon.
\end{equation}
Hence for a site $(I,0)$ on the $x$-axis, applying \eqref{2ndmom} and \eqref{2ndmom2} with $\theta=1/2$ we see that, $Q$-a.s.,
\begin{equation} \label{goodbound}
  Q(X_{(I,0)}=1|\mathcal{F}_{I}) \geq 1 - \epsilon,
\end{equation}
provided
\begin{equation} \label{ratiobound1}
\frac{Var_Q\Big( Z^{\beta, u, q}_{N, \tilde{\nu}^q_{(I,0)}}(\Omega_N^{\text{g}},\theta_{IN,0}(V))\ \big|\ \mathcal{F}_{I}\Big)} {\Big( E^Q\Big(Z^{\beta, u, q}_{N, \tilde{\nu}^q_{(I,0)}}(\Omega_N^{\text{g}},\theta_{IN,0}(V))\ \big|\ \mathcal{F}_{I}\Big)\Big)^2}\leq \frac{\epsilon}{8},\quad \text{g = forward, up}.
\end{equation}
Similarly, for $(I,J)$ with $J\geq1,$ we see that, $Q$-a.s.,
\begin{eqnarray}\label{goodbound2}
Q(X_{(I,J)}=1|\mathcal{F}_{I})\geq 1-\epsilon,
\end{eqnarray}
provided 
\begin{equation} \label{ratiobound2}
\frac{Var_Q\Big( Z^{\beta, u, q}_{N, \tilde{\nu}^q_{(I,J)}}(\Omega_N^{\text{g}},\theta_{IN,JN}(V))\ \big|\ \mathcal{F}_{I}\Big)} 
{\Big( E^Q\Big( Z^{\beta, u, q}_{N, \tilde{\nu}^q_{(I,J)}}(\Omega_N^{\text{g}},\theta_{IN,JN}(V))\ \big|\ \mathcal{F}_{I}\Big)\Big)^2} \leq \frac{\epsilon}{12},
\quad \text{g = up, forward, down}.
\end{equation}
For SSRW paths $s^1$ and $s^2,$ we have
\begin{equation} \label{overlaprole}
E^{Q}\left(e^{\beta H_N(s^1)+\beta u L_N(s^1)} e^{\beta H_N(s^2)+\beta u L_N(s^2) }\right)
=e^{\beta u L_N(s^1)} e^{\beta u L_N(s^2)} e^{\Phi(\beta)B_N(s^1, s^2)}e^{2\Lambda(\beta)N}.
\end{equation}
Recall $N=k_0M$.  
Using \eqref{LL}, \eqref{2case}, the Cauchy-Schwartz inequality and the fact that $(t-1)^2\leq t^2-1$ for $t\geq 1,$ for all $(I,J)$ we get $Q$-a.s.
\begin{align} \label{num}
Va&r_Q\Big( Z^{\beta, u, q}_{N, \tilde{\nu}^q_{(I,J)}}(\Omega_N^{\text{g}},\theta_{IN,JN}(V)) | \mathcal{F}_{I}\Big)\notag\\
&=e^{2\Lambda(\beta)N}\sum_{x,x'\in R(I,J)} \tilde{\nu}^q_{(I,J)}(x) \tilde{\nu}^q_{(I,J)}(x') E_{(x,x')}^{\otimes 2}\left(\Big(e^{\Phi(\beta) B_N(S^1,   
  S^2)}-1\Big)e^{\beta u L_N(S^1)} e^{\beta u L_N(S^2)}1_{ \Omega_N^{\text{g}} \times \Omega_N^{\text{g}} } \right)\notag\\
&\leq e^{2\Lambda(\beta) N}\sum_{x,x'\in R(I,J)}\Big[\tilde{\nu}^q_{(I,J)}(x)\tilde{\nu}^q_{(I,J)}(x')\Big(E_{(x,x')}^{\otimes 2}
  \left( e^{2\Phi(\beta) B_N(S^1, S^2)}-1\right)\Big)^{1/2}\notag\\
&\hskip 2in  \cdot \left(E_x e^{2\beta u L_N(S^1)} \right)^{1/2} \left(E_{x'}e^{2\beta u L_N(S^2)}\right)^{1/2}\Big]\notag\\
&\leq e^{2\Lambda(\beta) N}\left( \Big(E^{\otimes 2}_{(0,0)} e^{2\Phi(\beta) (B_M(S^1, S^2)+1)}\Big)^{k_0}-1\right)^{1/2} 
  \left( E_0e^{2\beta u(L_M+1)}\right)^{k_0}\notag\\
&=e^{2\Lambda(\beta)N}\left( \Big(E^{\otimes 2}_{(0,0)} \Big(e^{2\Phi(\beta)(B_M(S^1, S^2)+1)}-1\Big)+1\Big)^{k_0}-1\right)^{1/2} 
  \left( E_0e^{2\beta u(L_M+1)}\right)^{k_0}.
\end{align} 
For the denominator, by Lemma \ref{donsker}, for some $K_6>0$, $Q$-a.s.
\begin{align} \label{denom}
E^{Q}&\Big(Z^{\beta, u, q}_{N, \tilde{\nu}^q_{(I,J)}}(\Omega_N^{\text{g}},\theta_{IN,JN}(V))|\mathcal{F}_{I}\Big)\notag\\
&=\sum_{x\in R(I,J)}\tilde{\nu}^q_{(I,J)}(x)E^{Q}\Bigg(E_x\left[e^{\beta \sum_{k=1}^{N}\Big(v(IN+k, S_k)+u1_{S_k=0}\Big)}1_{\Omega_N^{\text{g}}}\right]\Bigg)\notag\\
&\geq e^{\Lambda(\beta)N} \sum_{x\in R(I,J)}\tilde{\nu}^q_{(I,J)}(x) P_x\left(\Omega_N^{\text{g}}\right)\notag\\
&\geq e^{\Lambda(\beta)N} K_6.
\end{align}
By Proposition \ref{Fprops}, we have $M=M(\beta u) \leq 5M(2\beta u)$ for small $\beta u$. Therefore
by Lemma~\ref{UPB}, for $K_2,K_3$ from that lemma,
\[E_0 e^{2\beta u(L_M+1)}\leq 6K_2e^{5K_3} =: K_7.\]
Combining this with \eqref{num} and \eqref{denom} we obtain that the left side of \eqref{ratiobound2} is bounded by
\begin{equation} \label{ratiobound3}
  K_6^{-2} K_7^{k_0}\left(\left( E_{(0,0)}^{\otimes 2}\left[e^{2\Phi(\beta)(B_M(S^1, S^2)+1)}-1\right]+1\right)^{k_0}-1\right)^{1/2}.
\end{equation}
Hence for our given $0<\epsilon<1$, we wish to apply Proposition~\ref{CRL} with
\begin{equation}
\label{AA}
R=M=\frac{1}{{\rm F}(\beta u)}, \quad a=\Big(\frac{K_6^4\epsilon^2}{12^2K_7^{2k_0}}+1\Big)^{1/k_0}-1;
\end{equation}
since $0<K_6<1$ and $K_7>1,$ we indeed have $a<1$ as needed.  From Proposition \ref{Fprops}(b), for $\beta$ small, provided $u\geq (2/K_4K_5(a))^{1/2}\beta$ we have $R\leq K_5\beta^{-4}$, so Proposition~\ref{CRL} does apply.
We then obtain from \eqref{ratiobound3} that the left side of \eqref{ratiobound2} (and also of \eqref{ratiobound1}) is bounded by $\epsilon/12$.  Thus \eqref{goodbound} and \eqref{goodbound2} hold, for $\beta$ and $\beta u$ small.

\subsection{Lipschitz Percolation}
\label{LipP}

Lipschitz percolation, the existence of open Lipschitz surfaces, was first introduced and studied in \cite{GHDD} and \cite{GH}. In this section, we briefly summarize and adapt some of their results for dimension $d=2,$ to use in our context. 

The independent site percolation model in $\mathbb{Z}^2$ is obtained by independently designating each site $x\in \mathbb{Z}^2$ \textit{open} with probability $p,$ otherwise \textit{closed}. The corresponding probability measure on the sample space $\Omega=\{0,1\}^{\mathbb{Z}^2}$ will be denoted by $\mathbb{P}_p,$ and expectation by $\mathbb{E}_p.$

Let $\mathbb{Z}_0^+=\{0, 1, 2, 3, \dots\}.$
A function $\mathcal{L}:\mathbb{Z}\rightarrow \mathbb{Z}_0^+$ is called \textit{Lipschitz} if 
for all $x, y \in \mathbb{Z}$ with $|x-y|=1$, we have $|\mathcal{L}(x)-\mathcal{L}(y)|\leq 1.$
$\mathcal{L}$ is called \textit{open} if for each $x \in \mathbb{Z},$ the site $(x, \mathcal{L}(x)) \in \mathbb{Z}^{2}$ is open. 

\begin{remark} In \cite{GHDD} and \cite{GH}, it was assumed that $\mathcal{L}\geq 1,$ but here it is more convenient to consider $\mathcal{L}(\cdot)\geq 0$, which of course does not change the results. 
\end{remark}

Let $A_{Lip}$ be the event that there exists an open Lipschitz function $\mathcal{L}:\mathbb{Z}\rightarrow \mathbb{Z}_0^+.$ 
Since $A_{Lip}$ is invariant under horizontal translation, we have $\mathbb{P}_p(A_{Lip})=0$ or $1.$ Since $A_{Lip}$ is also an increasing event, there exists a $p_L\in [0,1]$ such that 
\[ \mathbb{P}_p(A_{Lip})=\left\{
  \begin{array}{l l}
    0 & \quad \text{if}\ p<p_L,\\
    1 & \quad \text{if}\ p>p_L.
  \end{array} \right.\]
It was proved in \cite{GHDD} that $0<p_L<1$ for general dimension, but for the present 2-dimensional case Lipschitz percolation is a special type of oriented percolation, so standard contour arguments similar to (\cite{DR2} Section 10) suffice to show $p_L<1$.  
For any family $\mathcal{F}$ of Lipschitz functions, the lowest function \[\bar{\mathcal{L}}(x)=\inf\{\mathcal{L}(x): \mathcal{L} \in \mathcal{F}\}\]
is also Lipschitz. Hence if there exists an open Lipschitz function, then there exists a \textit{lowest open Lipschitz function}, and it will be again denoted by $\mathcal{L}.$  From \cite{GHDD}, $(\mathcal{L}(x): x \in \mathbb{Z})$
is stationary and ergodic.  

Let $D$ be the set of all $x\in \mathbb{Z}$ for which $\mathcal{L}(x)>0.$ 
Let $D_0$ be the connected component of $0$ in $D$, where connectedness is via adjacency in $\mathbb{Z}$. We define $D_0=\emptyset$ if $0\notin D.$

\begin{theorem}(\cite{GHDD},\cite{GH})
\label{GHT2}
Let $\mathcal{L}$ be the lowest open Lipschitz function. For $p>p_L,$ there exists $\alpha=\alpha(p)>0$ such that 
\[\mathbb{P}_p(\mathcal{L}(0)>n)\leq e^{-\alpha(n+1)}, \ n> 0.\]
There exists $p^{\prime}_{L}<1$ such that for $p\geq p^{\prime}_{L}$
\[\exp{(-\lambda n)}\leq \mathbb{P}_p(|D_0|\geq n)\leq \exp{(-\gamma n)}, \ n\geq 1,\]
where $\lambda=\lambda(p)$ and $\gamma=\gamma(p)$ are positive and finite. 
\end{theorem}

\begin{remark}
\label{IGP}
By Theorem~\ref{GHT2}, if the random field $X$ stochastically dominates independent site percolation of a sufficiently high density, then with positive probability there exists an infinite good path starting from $(0,0)$ in $\mathbb{L}_{CG}$.
\end{remark}

By Theorem \ref{GHT2}, for $p \geq p_L^\prime$ and $n \geq 1$ we have
\begin{align} \label{axisbond}
  1 - \mathbb{P}_p(\mathcal{L}(0)=\mathcal{L}(1)=0) &\leq \mathbb{P}_p(|D_0| > n) + \mathbb{P}_p((i,0) \text{ is closed for some } i \in (-n,n)) \notag\\
  &\leq e^{-\gamma(p)n} + (2n-1)(1-p).
\end{align}
We may assume $\gamma(p)$ is nondecreasing in $p$.  Then given $\epsilon>0$, we can first apply \eqref{axisbond} with $p=p_L^\prime$, and choose $n$ large enough so $e^{-\gamma(p_L^\prime)n} < \epsilon/2$.  Then for $p$ sufficiently close to 1, both terms on the right side of \eqref{axisbond} are bounded by $\epsilon/2$, so by the ergodic theorem, 
\begin{equation} \label{ergodic}
  \lim_{N\to \infty}\frac{1}{N}\sum_{i=1}^N1_{(\mathcal{L}(i-1)=\mathcal{L}(i)=0)}=\mathbb{P}_p(\mathcal{L}(0)=\mathcal{L}(1)=0) 
    > 1-\epsilon,\ \ \mathbb{P}_p-\text{a.s.}
  \end{equation}

\subsection{Stochastic Domination}
To obtain the domination referenced in Remark \ref{IGP} we will need the following result of Liggett, Schonmann and Stacey \cite{LSS}.

\begin{theorem}
\label{SD} 
Let $(X_s)_{s\in \mathbb{Z}}$ be a collection of $0$-$1$ valued $k$-dependent random variables, and suppose that there exists a $p\in(0,1)$ such that for each $s\in \mathbb{Z}$
\[\mathbf{P}(X_s=1)\geq p.\] 
Then if 
\[p> 1-\frac{k^k}{(k+1)^{k+1}},\]
then $(X_s)_{s\in \mathbb{Z}}$ is dominated from below by a product random field with density $0<\rho(p)<1.$ Furthermore, $\rho(p)\to 1$ as $p\to 1.$
\end{theorem}

Fix $\epsilon>0$ and choose $p<1$ so that an open Lipschitz function exists a.s.~and \eqref{ergodic} holds.  Then choose $\eta$ with $\rho(1-\eta)>p$ (with $\rho(\cdot)$ from Theorem \ref{SD}.)  For fixed $I\geq 1,$ the boxes $B(I,J), B(I,J^\prime)$ are disjoint for $|J - J^\prime| >4$, so conditionally on $\mathcal{F}_{I}$, 
$\{X_{(I,J)}: 0 \leq J \leq I\}$ is a $4$-dependent collection of random variables.  From \eqref{goodbound} and \eqref{goodbound2}, for sufficiently small $\beta u>0$ and $\beta>0$ with $u\geq K_8(\eta)\beta,$
\[Q(X_{(I,J)}=1|\mathcal{F}_{I})\geq 1-\eta\ \ \ Q-a.s. \ \ \text{for each}\ I\geq 1, J\geq 0.\]
We can apply Theorem~\ref{SD} inductively on $I$ to see that there exists a collection of i.i.d.~$0$-$1$ valued random variables $\{Y_{(I,J)}: (I,J)\in \mathbb{L}_{CG}\}$ with $Q(Y_{(I,J)}=1)=\rho(1-\eta)$ and
\begin{equation}
\label{conditional}
Q(X_{(I,J)}\geq Y_{(I,J)}|\mathcal{F}_{I})=1\ \ Q-a.s.
\end{equation}
and therefore also unconditionally, $X(I,J) \geq Y(I,J)$ a.s.  It follows that the configurations $\{X_{(I,J)}:(I,J) \in \mathbb{L}_{CG}\}$ and $\{Y_{(I,J)}:(I,J) \in \mathbb{L}_{CG}\}$ also a.s.~have lowest open Lipschitz functions $\mathcal{L}_X\leq\mathcal{L}_Y$ . With positive probability we have $\mathcal{L}_X(0)=0$, in which case $\mathcal{L}_X=\Gamma^{G,\infty}$ is the infinite good path from $(0,0)$.

\subsection{Final Steps}
Let 
\[
  R_L := \sum_{I=1}^L 1_{\{\Gamma^{G,\infty}(I-1) = \Gamma^{G,\infty}(I)=0\}}.
\]
Since $\Gamma^{G,\infty} \leq \mathcal{L}_X \leq \mathcal{L}_Y$ on $\mathbb{Z}_0^+$, it follows from \eqref{ergodic} applied to $\mathcal{L}_Y$ that when $\Gamma^{G,\infty}$ exists,
%By ergodicity and \eqref{ergodic}, the limit 
\begin{equation} \label{alphadef}
  \alpha=\alpha(\beta u):=\liminf_{L \to \infty}\frac{R_L}{L} > 1-\epsilon.
  \end{equation}
Recall that
\[U_{\text{off}}= \Theta_{\rm off}e^{\Lambda(\beta)N}, \quad U_{\text{on}} = \Theta_{\rm on} e^{(\Lambda(\beta) + (1-\epsilon)\mathrm{F}(\beta u))N},\]
where 
\begin{equation}
\label{asympt}
  \Theta_{\rm on} = \Theta_{\rm on}(\epsilon) = \frac{\epsilon_0}{4}\mathbf{P}(\xi\geq  \frac{1}{4\sqrt{\epsilon}})
    \sim \frac{\epsilon_0\sqrt{\epsilon}}{\sqrt{2\pi}}\, e^{-1/32\epsilon} \ \text{as} \ \epsilon \to 0
\end{equation}
and $\Theta_{\rm off}$ is the minimum of $\Theta_{\rm on}$ and a constant.
Define $\Theta_0=\Theta_0(\epsilon)=-(\alpha \log \Theta_{\rm on}+(1-\alpha)\log \Theta_{\rm off})>0.$
For some $K_9>0$ we have
\begin{equation}\label{theta3}
\Theta_0(\epsilon) \leq \frac{K_9}{\epsilon}, \quad \epsilon \in (0,1).
\end{equation}

For $L\geq 1$ when an infinite good path from $(0,0)$ exists we have
\[\frac{1}{LN}\log Z^{\beta, u, q}_{LN} \geq \frac{1}{LN}\log  Z^{\beta, u, q}_{LN}(\Omega_{0\to L}^{G,\infty}) \] 
and using \eqref{samepath},
\begin{equation}
\label{PrZ}
Z^{\beta,u,q}_{LN}(\Omega_{0\to L}^{G,\infty}) = \prod_{I=1}^L\frac{Z^{\beta,u,q}_{IN}(\Omega_{0\to I}^{G,\infty})}
  {Z^{\beta,u,q}_{(I-1)N}(\Omega_{0\to I-1}^{G,\infty})}
  =\prod_{I=1}^L Z^{\beta, u, q}_{N, \tilde{\nu}^q_{(I-1,\Gamma^{\text{G},\infty}(I-1))}}
  (\Omega_{I-1\to I}^{G,\infty},\theta_{I-1,\Gamma^{\text{G},\infty}(I-1)} V)
\end{equation}
where $Z^{\beta,u,q}_0:=1.$  Note that \eqref{samepath} also guarantees that the measures $\tilde{\nu}^q_{(I-1,\Gamma^{\text{G},\infty}(I-1))}$ on the right side of \eqref{PrZ} are the ones used in the definition of open/closed coarse-grained sites.

Let $p_{0,\infty}>0$ be the probability that there is an infinite good path from $(0,0)$ in the configuration $X$.  When such a path exists, by \eqref{PrZ}  we have for all $L \geq 1$
\begin{equation} \label{ZU}
  Z^{\beta, u, q}_{LN} \geq Z^{\beta,u,q}_{LN}(\Omega_{0\to L}^{G,\infty}) \geq U_{\text{on}}^{R_L} U_{\text{off}}^{L-R_L}.
\end{equation}
Therefore 
\begin{eqnarray*}
Q\Big(\frac{1}{LN}\log Z^{\beta, u, q}_{LN} &\geq &\frac{1}{LN}\log U_{\text{on}}^{R_L} U_{\text{off}}^{L-R_L}
  \text{ for all } L \geq 1 \Big)\geq p_{0,\infty}.
\end{eqnarray*}
Since the quenched free energy is self-averaging, recalling $N=k_0M = k_0/\mathrm{F}(\beta u),f^a(\beta,u)=\mathrm{F}(\beta u)+\Lambda(\beta)$ and $U_{\rm off} \leq U_{\rm on}$, using \eqref{theta3} we get 
\begin{align} \label{fqlower}
f^q(\beta, u) &\geq \alpha \frac{1}{N}\log U_{\text{on}} + (1-\alpha)\frac{1}{N}\log U_{\text{off}}\notag\\
&=\alpha \left( (1-\epsilon)\mathrm{F}(\beta u) + \Lambda(\beta)\right)-\frac{1}{N}\Theta_0+(1-\alpha)\Lambda(\beta)\notag\\
&\geq \Lambda(\beta) + \alpha (1-\epsilon)\mathrm{F}(\beta u)-\frac{K_9}{k_0\epsilon}\mathrm{F}(\beta u).
\end{align}
By choosing $k_0 = \lfloor{K_9\epsilon^{-2}+1}\rfloor$, we make the third term on the right side of \eqref{fqlower} greater than $-\epsilon \mathrm{F}(\beta u)$.
This and \eqref{alphadef} show that
\begin{equation}\label{freelower}
  f^q(\beta, u)\geq \Lambda(\beta) + (1-3\epsilon)\mathrm{F}(\beta u) > \Lambda(\beta)= f^a(\beta,0) \geq f^q(\beta,0),
  \end{equation}
proving \eqref{samefree} and \eqref{critbound}.

\section{Proof of Theorem \ref{beta4}.}
We describe here the necessary modifications to the proof of Theorem \ref{MainTh}. We need only prove the upper bound, as the lower bound is proved in \cite{L}. 

In place of separate ``on'' and ``off'' constants, we use simply (cf. \eqref{offx})
\[
  \Theta = \frac{1}{4}\min(A^{\rm{forward}},A^{\rm{up}},A^{\rm{down}}), \quad
    U := \Theta e^{\Lambda(\beta)N}.
\]
A site $(I,J)$ is now called \emph{open} if (cf. \eqref{opensite})
\begin{equation} \label{opensite2}
Z^{\beta, u, q}_{N, \tilde{\nu}^q_{(I,J)}}(\Omega_N^{\text{g}},\theta_{IN,JN}(V))\geq U,  \quad \text{g = up, forward, down}.
\end{equation}
%Given $\epsilon>0$ we take $a$ as in \eqref{AA}, and $M=K_5(a)\beta^{-4}$.
Given $\epsilon>0$ we obtain $a$ as in \eqref{AA} and then $K_5(a)$ from Proposition \ref{CRL}, and take $\tilde M=K_5(a)\beta^{-4}$.  We otherwise repeat the proof of Theorem \ref{MainTh} but with $u=0$ and $\tilde M$ in place of $M$ throughout, and $k_0=1$ so that $N=\tilde M$.  The density of open sites can be made arbitrarily close to 1 by taking $\epsilon$ small, and then there is a positive probability that an infinite good path exists.  In that case we have the lower bound (cf. \eqref{ZU})
\[
  Z_{LN}^{\beta,u,q} \geq U^L, \quad L \geq 1,
\]
which as in \eqref{fqlower} yields
\[
  f^q(\beta,0) \geq \frac{1}{N}\log U = f^a(\beta,0) - \frac{1}{N}\log \frac{1}{\Theta} \geq f^a(\beta,0) - C_2\beta^4,
\]
concluding the proof.

\section{Acknowledgements.}
The authors would like to thank S. R. S. Varadhan for the proof of Lemma \ref{hayati}, and an anonymous referee who pointed out that the proof of Theorem \ref{MainTh} also yielded Theorem \ref{beta4}.

\end{document}